\documentclass[twoside,centered,14pt,reqno]{article}
\usepackage{lineno,hyperref}

\usepackage{amssymb,latexsym,epsfig,amsmath}
\usepackage{amsthm}
\usepackage{color}
\usepackage{subfig}
\usepackage{rotating}
\usepackage{mathtools}
\usepackage{graphicx}

\usepackage{bm}
\usepackage{multirow}

\usepackage{mathtools}
\usepackage{multirow, longtable}

\newtheorem{te}{Theorem}[section]

\newtheorem{definition}[te]{Definition}
\newtheorem{os}[te]{Remark}
\newtheorem{prop}[te]{Proposition}

\newcommand{\be}{\begin{equation}}
\newcommand{\ee}{\end{equation}}

\modulolinenumbers[5]
\begin{document}

\vskip 0.50truecm
\font\title=cmbx14 scaled\magstep2
\font\bfs=cmbx12 scaled\magstep1
\font\little=cmr10
\begin{center}
{\title {\Large \bf {The tempered space-fractional
 \\ Cattaneo equation}}}
 \\ [0.25truecm]
{\Large Luisa  Beghin$^1$. Roberto Garra$^2$, \\  Francesco  Mainardi$^{3,(\star)}$, 
Gianni Pagnini$^{4.5}$}
\\[0.25truecm]
$\null^1${\little Department  of Statistical Sciences, Sapienza University of Rome,}
\\ {\little  P.le Aldo Moro 5, I-00185 Rome, Italy}
\\[0.25truecm]
$\null^2$ {\little Institute of Marine Sciences, National Research Council (CNR),}
\\{\little  Via del Fosso del Cavaliere, I-00133 Rome, Italy}
\\[0.25truecm]
$\null^3$ {\little Department of Physics and Astronomy, University of Bologna, and INFN}
\\  {\little Via Irnerio 46, I-40126 Bologna, Italy}
\\
$\null^{(\star)}$ {\little Corresponding Author.   E-mail: mainardi@bo.infn.it}
\\[0,25truecm]
$\null^4$ {\little BCAM (Basque Center for Applied Mathematics)}
\\{\little , Alameda de Mazarredo 14, E-48009 Bilbao, Basque Country, Spain}
\\[0,25truecm]
$\null^4$ {\little Ikerbasque (Basque Foundation for Science),}
\\{\little Plaza Euskadi 5, E-48009 Bilbao, Basque Country, Spain}

\end{center}

\begin{abstract}
We consider the time-fractional Cattaneo equation involving
the tempered Caputo space-fractional derivative. 
We find the characteristic function of the related process and 
we explain the main differences with previous stochastic treatments 
of the time-fractional Cattaneo equation.

\vspace{0.4cm}

\noindent {\bf Mathematics Subject Classification:} 
26A33, 33E12, 35C05, 44A10, 60H30										

\vspace{0.3cm}

\noindent {\bf Keywords:} 
Cattaneo equation;
tempered fractional derivative;
stochastic processes

\end{abstract}



 \section{Introduction}
The recent experimental observation 
of wavelike thermal transport in graphite at temperatures above 
$100 \, {\rm K}$ \cite{huberman_etal-s-2019}, 
which confirms the just derived theoretical prediction \cite{ding_etal-nl-2018},
solves the long-standing challenge to establish the existence, 
in certain materials, of phonon hydrodynamics and second sound phenomenon 
at relatively high temperatures, see e.g \cite{lindsay_etal-jap-2019},
\cite{lee_li-2020}. 
In fact,
the occurrence of second sound was previously limited 
to a handful of materials at low temperatures and 
therefore the scientific and practical significance of this phenomenon
was limited. This new experimental evidence indeed potentially indicate 
an important role of second sound in microscale transient heat transport 
in two-dimensional and layered materials in a wide temperature range. 

A wavelike thermal transport implies a phonon hydrodynamics regime 
that is intermediate between ballistic and diffusive regimes, 
and it is properly described by a generalization of Fourier's law 
into the viscous heat equation (or damped wave equation) 
\cite{simoncelli_etal-prx-2020}.
Thus, second sound in solids occurs when the
local temperature follows an hyperbolic equation analogue to 
the telegrapher's equation in electromagnetism and to Cattaneo's equation 
in conduction problems \cite{hardy-prb-1970} that reads
\begin{equation}
\frac{\partial^2 f}{\partial t^2} +
\frac{\partial f}{\partial t} =
\frac{\partial^2 f}{\partial x^2} \,.
\label{TE}
\end{equation}
In this respect, we remind that when the same 
hyperbolic equation governs the pressure or the density
then we have the first sound.
In general, for second sound phenomena the damping term
dominates the inertial term and we have the diffusion equation,
while for the first sound phenomena the opposite is true and we
have the wave equation.	

This experimental ascertainment at relative high temperature 
of the second sound in graphen, and in general in solids, 
motivated us to a mathematical investigation 
of non-local extensions of the viscous heat equation 
(\ref{TE}) in the spirit 
of taking into account the already widely established
evidences of non-local effects both in diffusive, 
see, e.g., \cite{klm}, and viscous, see, e.g. \cite{mainardi-2010}, systems.  
Since from the physical point of view, 
it is more appropriate to speak about telegrapher's equation 
when we speak of applications to electromagnetism, 
while it is common to use the name Cattaneo equation 
in problems regarding heat conduction and anomalous transport processes,
then hereinafter we refer to the Cattaneo equation,
which, we remind, still calls for a deep mathematical analysis 
both in the classical formulation 
\cite{spigler-mmas-2020,carillo_etal-mcs-2020} 
and in the non-local extension 
\cite{ferrillo_etal-siamjam-2018,angelani_etal-jpa-2020}.

The Cattaneo equation plays indeed a relevant role in many different physical 
contexts. In particular, in random motion and heat propagation models
with finite front velocity \cite{Preziosi} 
and, more recently, also in run-and-tumble bacterial dynamics \cite{luca}. 
The space- and time-fractional counterpart of this equation has gained 
a relevant interest both for physical applications 
and stochastic models related to continuous-time persistent random walks
\cite{gianni,mas}. 
Non-local generalization of the Cattaneo equation
through the fractional calculus has been already investigated,
see e.g., \cite{compte_etal-jpa-1997,metzler_etal-pa-1999},
\cite{Enzo}, \cite{mirko}, and more recently \\
\cite{gianni,mas}. 
Here, the novel contribution with respect to the literature,
lays on the analysis of the role of the tempering of fractional derivative 
when applied for generalizing the space-derivative. 
We remind that, indeed, the tempered fractional diffusion equation 
has been already investigated,
e.g., \cite{liemert_etal-fcaa-2017,lischke_etal-fcaa-2019}.

The rest of the paper is organized as follows. 
First, we provide in Section 2 the preliminaries notions
on the tempered derivatives and the related processes,
and later in Section 3 the main results are reported.

\section{Preliminaries on tempered fractional derivatives and 
related stochastic processes}
The \textit{shifted fractional derivative} 
has been used in the physical literature for mathematical models of 
wave propagation in porous media \cite{hany} 
and in probability in relation with 
the Tempered Stable Subordinator (TSS).
The shifted fractional derivative is defined as 
\begin{equation}
\left(\lambda + \frac{d}{dx}\right)^\alpha f(x) = 
e^{-\lambda x} D_x^\alpha \, [e^{\lambda x}f(x)] \,, 
\quad \alpha\in (0,1) \,, \quad \lambda \geq 0 \,, \quad  x \geq 0 \,,
\end{equation}
where $D_x^\alpha$ denotes the space fractional derivative 
in the sense of Caputo of order $\alpha \in (0,1)$,
i.e.,
\begin{equation}
D_x^\alpha f(x) = \frac{1}{\Gamma(1-\alpha)}\int_0^x (x-\xi)^{-\alpha}
\frac{\partial f}{\partial \xi} \, d\xi \,.
\end{equation}
Thus, the Laplace transform of the shifted fractional derivative 
is given by \cite{hany}
\begin{equation}\label{lt}
\mathfrak{L}	
\bigg\{\left(\lambda+\frac{d}{dx}\right)^\alpha f(x)\bigg\}(s) = 
(s+\lambda)^\alpha \widetilde{f}(s) - (s+\lambda)^{\alpha-1} f(0^+)
\,,
\end{equation}
where we denoted by $s$ the Laplace parameter and $\widetilde{f}(s)$ 
the Laplace transform of the function $f(x)$, i.e.,
$\displaystyle{\int_0^\infty e^{-sx}f(x) \, dx=\widetilde{f}(s)}$.

In probability the transition density $f_{\lambda, \alpha}(x,t)$ 
of the TSS can be introduced by "tempering" the transition density 
of the $\alpha$-stable subordinator $h_\alpha(x,t)$ as follows
\begin{equation}
\label{1.3}
f_{\lambda, \alpha}(x,t):= e^{-\lambda x+\lambda^\alpha t} h_\alpha(x,t)
\,.
\end{equation}
Indeed, it is possible to prove \cite{luisa} that the density of 
the TSS satisfies the following tempered-fractional equation
\begin{equation}
\frac{\partial}{\partial t}f_{\lambda, \alpha}(x,t) = 
\lambda^\alpha
f_{\lambda, \alpha}(x,t) - 
\left(\lambda + \frac{\partial}{\partial x}\right)^\alpha 
f_{\lambda, \alpha}(x,t) = 
- \frac{\partial^{\lambda,\alpha}}{\partial x^{\lambda,\alpha}}
f_{\lambda, \alpha}(x,t) \,,
\end{equation}
where
\begin{equation}
\frac{\partial^{\lambda,\alpha}}{\partial x^{\lambda,\alpha}}f(x,t):= 
\left(\lambda + \frac{\partial}{\partial x}\right)^\alpha f(x,t)
- \lambda^\alpha f(x,t) \,,
\end{equation}
under the conditions
\begin{equation}
f_{\lambda, \alpha}(x,0) = \delta(x) \,, \quad f_{\lambda, \alpha}(0,t)= 0 \,.
\end{equation}

We here briefly recall the notion of fractional tempered stable (TS) 
process \cite{luisa}.
We denote by $\mathcal{L}_\alpha(t):=\inf\{s:\mathcal{H}_\alpha(s)>t\}$, 
$t\geq 0$, $\alpha\in(0,1)$, 
the inverse of the $\alpha$-stable subordinator $\mathcal{H}_\alpha(t)$,
then

\begin{definition}
	Let $\mathcal{L}_\nu(t)$, with $t\geq 0$, 
be the inverse of the stable subordinator, 
then the fractional TS process is defined as
	\begin{equation}
		\mathcal{T}^\nu_{\lambda, \alpha}(t):=
\mathcal{T}_{\lambda, \alpha}(\mathcal{L}_\nu(t)) \,, 
\quad t\geq 0 \,, \quad \lambda \geq 0 \,, \quad \nu \,, \alpha \in (0,1) \,,
	\end{equation}
	where $\mathcal{L}_\nu$ is independent of 
the tempered stable subordinator (TSS) $\mathcal{T}_{\lambda, \nu}$.
\end{definition}

The density of the fractional TS process 
$\mathcal{T}^\nu_{\lambda, \alpha}(t)$ satisfies
the tempered fractional equation \cite[Theorem 6]{luisa} 
\begin{equation}\label{in0}
D_t^\nu f = 
\bigg[\lambda^\alpha-\left(\lambda+\frac{\partial}{\partial x}\right)^{\alpha}
\bigg]f \,, 
\quad \lambda\geq 0 \,,
\quad \nu \,, \alpha \in (0,1) \,, 
\end{equation}
under the initial-boundary conditions
\begin{equation}
	\begin{cases}
		f(x,0)= \delta(x) \,,\\
		f(0,t) = 0 \,.
	\end{cases}
\end{equation}
Finally, we recall also that the density 
of the time-changed Brownian motion
\begin{equation}\label{1.7}
\mathcal{X}^{\nu}_{\lambda,\alpha}(t) := 
B(\mathcal{T}^\nu_{\lambda, \alpha}(t)) \,,
\quad \lambda\geq 0 \,,
\quad \nu , \alpha \in (0,1) \,,
\end{equation}
coincides with the solution of the tempered equation \cite{luisa}
\begin{equation}
\label{in1}
D_t^\nu g = 
\bigg[\lambda^\alpha-\left(\lambda-\frac{\partial^2}{\partial x^2}
\right)^{\alpha}\bigg] \, g \,,
\quad
- \infty < x < + \infty \,.
\end{equation}

\section{The tempered space-fractional Cattaneo-type equation}
Let $\mathcal{L}^\beta(t)$, with $t>0$, 
be the inverse process of the sum of two independent positively 
skewed stable random variables $H_1^{2\beta}$ and $H_2^\beta$, that is
\begin{equation}
\mathcal{L}^\beta(t) 
:= \inf \bigg\{s\geq 0, \ H_1^{2\beta}(s)+(2k)^{1/\beta}H_2^\beta(s) \geq t
\bigg\}\,, \quad t \,, k >0 \,, \quad \beta \in (0,1/2) \,. 
\end{equation} 
We recall that the Laplace transform with respect to $t$
of the law $l_\beta(x,t)$ 
of the process $\mathcal{L}^\beta(t)$ is given by \cite{mirko}
\begin{equation}
	\widetilde{l}_\beta(x,s)= 
(s^{2\beta-1}+2k s^{\beta-1})e^{-xs^{2\beta}-2k xs^\beta} \,,
\end{equation}
and satisfies the fractional equation
\begin{equation}
D_t^{2\beta} u +2k D_t^\beta u = - \frac{\partial u}{\partial x}
\,.
\end{equation}
Then, we have the following result
\begin{te}
The solution of the tempered fractional equation 
\begin{equation}\label{tem2}
D_t^{2\beta} f+2k D_t^{\beta}f =
\bigg[
\lambda^\alpha-\left(\lambda-\frac{\partial^2}{\partial x^2}\right)^{\alpha}
\bigg] \, f \,, 
\quad -\infty < x < +\infty \,,
\end{equation} 
with $\beta \in (0,1/2)$ and $\alpha \in(0,1)$,
under the conditions $u(x,0) = \delta(x)$ and $u(0,t) = 0$,
coincides with the probability law of the process
\begin{equation}\label{tc}
W(t): =  B(\mathcal{T}_{\lambda,\alpha}(\mathcal{L}^\beta(t))) \,, 
\quad t>0 \,.
\end{equation}
Moreover, the fundamental solution of equation \eqref{tem2} 
has the following Fourier transform with respect to $x$
\begin{equation}
\widehat{u}(\xi,t) = \frac{1}{2}\bigg[
\bigg(1+\frac{k}{\sqrt{k^2-\theta(\xi)}}\bigg)E_{\beta,1}(r_1t^\beta)+
\bigg(1-\frac{k}{\sqrt{k^2-\theta(\xi)}}\bigg)
E_{\beta,1}(r_2t^\beta)\bigg] \,,
\end{equation}
with 
\begin{align}
\nonumber & \theta (\xi) = (\lambda+\|\xi\|^2)^\alpha-\lambda^\alpha \,,\\
		\nonumber & r_1 = -k+\sqrt{k^2-\theta (\xi)} \,,\\
		\nonumber & r_2 = -k-\sqrt{k^2-\theta (\xi)} \,,
	\end{align}
and where 
$$
E_{\beta, \gamma}(t) = \sum_{k=0}^\infty\frac{t^k}{\Gamma(\beta k+\gamma)} \,,
$$
is the  Mittag--Leffler function, for $\beta >0$ and $\gamma \in \mathbb{C}$.
\end{te}

\begin{proof}
	The probability law of the process $W(t)$ is given by
	\begin{equation}
		w(x,t)  = \int_0^\infty v(x,\mu)l_\beta(\mu, t) \, d\mu \,,
	\end{equation}
where $v(x,t)$ is the density of the tempered stable subordinator whose 
Fourier-transform is given by 
\begin{equation}
\widehat{v}(\xi,t) = e^{-t \theta(\xi)} \,.
\end{equation}
Therefore, the Fourier transform of $w(x,t)$ is given by
\begin{equation}
	\widehat{w}(\xi, t) = 
\int_0^\infty e^{-\mu \theta(\xi)}l_\beta(\mu, t) \, d\mu \,.
\end{equation}
Since the Laplace transform with respect to $t$
of $l_\beta(x,t)$ is given by 
\begin{equation}
	\widetilde{l}_\beta(x,s) = 
(s^{2\beta-1}+2k s^{\beta-1})e^{-xs^{2\beta}-2k xs^\beta} \,,
\end{equation}
then the Fourier--Laplace transform of the probability law 
of the process $W(t)$ is given  by
\begin{align}
\nonumber \widehat{\widetilde{w}}(\xi,s) &= 
\int_0^\infty e^{-st} \, dt 
\int_0^\infty e^{-\mu \theta(\xi)}l_\beta(\mu, t) \, d\mu\\
& = (s^{2\beta-1}+2k s^{\beta-1})\int_0^\infty e^{-\mu \theta(\xi)-\mu s^{2\beta}-2k \mu s^\beta} d\mu = \frac{s^{2\beta-1}+2k s^{\beta-1}}{s^{2\beta}+2k s^\beta+\theta(\xi)}
\,.
\label{ltf} 
\end{align}
If we compare it with the Fourier-Laplace transform of the 
fundamental solution of the equation \eqref{tem2}, 
we observe that they are iqual and the claimed result holds. 

Regarding the inverse time-Laplace transform of \eqref{ltf},
we report that it can be obtained through an algebraic manipulations 
\cite[pp. 1021--1022]{mirko} and by recalling the Laplace transform 
formulas for the Mittag--Leffler functions. 
\end{proof}

We observe that the process $X(t):=B(\mathcal{T}_{\lambda,\alpha}(t))$ 
is indeed a L\'evy process with L\'evy exponent
\begin{align}
\nonumber 	\psi(\xi) &= -\frac{1}{t}\ln \mathbb{E}[e^{i\xi X(t)}]\\
\nonumber  & = -\frac{1}{t}\ln \mathbb{E}[\mathbb{E}[e^{i\xi X(\mathcal{T}(t))}|\mathcal{T}(t)]] = -\ln\mathbb{E} e^{-\xi^2\mathcal{T}(t)/2}\\
\nonumber  & = - \frac{1}{t} \ln 	\bigg[\int_0^\infty e^{-\frac{\xi^2}{2}x} f_{\lambda, \alpha}(x,t)dx\bigg]
	\,,
\end{align}
and by using equation \eqref{1.3} we have that 
\begin{align}
	\nonumber 	
\psi(\xi) &=  - -\frac{1}{t}\ln \bigg[\int_0^\infty e^{\lambda^\alpha t-\lambda x-\frac{\xi^2}{2}x} h_{\alpha}(x,t) \, dx\bigg]\\
	&=  -\frac{1}{t}\ln[e^{\lambda^\alpha t-(\frac{\xi^2}{2}+\lambda)^\alpha t}] = 
	\bigg[(\frac{\xi^2}{2}+\lambda)^\alpha-\lambda^\alpha\bigg]
\,.
\end{align}
Since the Laplace exponent of the sum of stable subordinators 
$H_1^{2\beta}(t)+(2k)^{1/\beta}H_2^\beta(t)$ is given by
\begin{equation}
	\phi(s) = (s^{2\beta}+2k s^\beta) \,,
\end{equation}
then we have that the inverse process $\mathcal{L}^\beta(t)$ 
has L\'evy exponent given by 
\begin{equation}
\mathfrak{L}\bigg\{\mathbb{E}\mathcal{L}^\beta(t);s\bigg\} = \frac{1}{s\phi(s)} = \frac{1}{s^{2\beta+1}+2k s^{\beta+1}}
\,,
\end{equation}
whose inverse Laplace transform, namely $U(t)$, is given by
\begin{equation}
	U(t):= t^{2\beta}E_{\beta,2\beta+1}(-2kt^\beta)
\,.
\end{equation}

We recall now \cite[Theorem 2.1]{Leo} 
that if we have a time-changed L\'evy process $X(Y(t))$, 
with X(t) an homogeneous L\'evy process and 
Y(t) a non-decreasing process independent of X, 
then we have
\begin{equation}
	\mathbb{E}X(Y(t)) = U(t) \mathbb{E}X(1) \,,
\end{equation}
and 
\begin{equation}
	Var X(Y(t)) = \mathbb{E}[X(1)]^2 Var[Y(t)]+U(t)Var[X(1)]
\,,
\end{equation}
where $U(t) = \mathbb{E}Y(t)$.
In our case, since the mean value of $X(t)$ is null for all $t>0$, 
we have that 
\begin{equation}
\mathbb{E}X(\mathcal{L}^\beta(t)) = 0 \,, \quad {\rm for} \,
{\rm all} \quad t >0 \,,
\end{equation}
and 
\begin{equation}
	Var X(\mathcal{L}^\beta(t)) =  
t^{2\beta}E_{\beta,2\beta+1}(-2kt^\beta)Var[X(1)]
\,.
\end{equation}
We observe that
\begin{align}
	\nonumber Var[X(1)] &=\mathbb{E}\bigg[\mathbb{E}[B(\mathcal{T}_{\lambda,\alpha}(1))^2|\mathcal{T}_{\lambda,\alpha}(1)]\bigg]= \mathbb{E}\bigg[\mathcal{T}(1)^2]\bigg]\\
	\nonumber& = \int_0^{+\infty} z^2 e^{-\lambda z+\lambda^\alpha}h_\alpha(z,1) \, dz= e^{\lambda^\alpha}\frac{d^2}{d\lambda^2}\int_0^{+\infty}e^{-\lambda z} h_\alpha(z,1) \, dz\\
	& = e^{\lambda^\alpha}\frac{d^2}{d\lambda^2}e^{-\lambda^\alpha}= \alpha \lambda^{\alpha-2}[1-\alpha+\alpha\lambda^\alpha]
\end{align}
and we conclude that 
\begin{equation}
	Var X(\mathcal{L}^\beta(t)) =  \alpha \lambda^{\alpha-2}[1-\alpha+\alpha\lambda^\alpha]t^{2\beta}E_{\beta,2\beta+1}(-2kt^\beta) \,.
\end{equation}

We recover, for $\lambda = 0$, 
the result for the space-time fractional (non-tempered) 
Cattaneo process
\cite[Theorem 4.1]{mirko} 
\begin{equation}
	Var X(\mathcal{L}^\beta(t)) =  
\alpha [1-\alpha]t^{2\beta}E_{\beta,2\beta+1}(-2kt^\beta)
\,.
\end{equation}

\begin{os}
We observe that the probabilistic interpretation given by the time-changed process \eqref{tc} works only for $\beta\in (0,1/2)$, that is a sort of multi-term time-fractional diffusion equation with space-tempered derivatives. On the other hand, the analytical representation of the solution is correct also for $1/2\!< \beta\!<\!1$, under the additive constraint 
$\partial_t u\big|_{t = 0}=0$. Moreover, we recover the Fourier transform of the fundamental solution that was 
originally found by Beghin and Orsingher \cite{Enzo}.
\end{os}

We can also consider the space-Laplace transform of the solution for the more general case $\alpha \in (0,1)$, even if in this case we loose the probabilistic representation that is valid only in the case $\alpha\in (0,1/2)$.

For the particular case $\beta = 1$, we have the following 
\begin{prop}
	The space-Laplace transform of the solution for the fractional problem \eqref{tem2}, under the conditions
	$u(x,0) = \delta(x)$ and $\partial_t u(x,t)\bigg|_{t=0}=0$ is given by
 \begin{align}\label{sol0}
	\tilde{u}(s,t) = \frac{e^{-kt}}{2}&\bigg[\left(1+\frac{k}{\sqrt{k^2-\psi(s)}}\right)e^{t\sqrt{k^2-\psi(s)}}\\
	\nonumber	&+\left(1-\frac{k}{\sqrt{k^2-\psi(s)}}\right) e^{-t\sqrt{k^2-\psi(s)}}\bigg] \,,
\label{cf}
\end{align}
where 
\begin{equation}
	\psi(s) = (s+\lambda)^\alpha-\lambda^\alpha \,.
\end{equation}
\end{prop} 

\begin{proof}
We take the space-Laplace transform and, by using \eqref{lt}, 
we have that
    \begin{equation}
    \frac{\partial^2 \widetilde{u}}{\partial t^2}+
2 k\frac{\partial \widetilde{u}}{\partial t} = 
(s+\lambda)^\alpha \widetilde{u}-\lambda^\alpha \widetilde{u} = 
\psi(s) \widetilde{u}
\,,
\end{equation}
whose solution, under the given conditions, is given by \eqref{sol0}.
\end{proof}

To conclude, we consider the following Dirichlet problem 
\begin{equation}\label{dr}
	\begin{cases}
	& \displaystyle \frac{\partial^2 u}{\partial t^2}+ 
2 k\frac{\partial u}{\partial t} = 
\frac{\partial^{\lambda, \alpha}u}{\partial x^{\lambda, \alpha}} \,, 
\quad x \geq 0 \,,\\
& u(x,0) = 0 \,, \quad u(0,t) = \phi(t) \,,\\
& \displaystyle \frac{\partial u}{\partial t}\bigg|_{t=0} = 0 \,,
	\end{cases}
\end{equation}
We have the following 
\begin{prop}
The time-Laplace transform of the solution for the Dirichlet problem 
\eqref{dr} is given by
\begin{equation}
	\widetilde{u}(x,s) = \widetilde{\phi}(s) e^{-\lambda x} 
        E_{\alpha,1}\left[-
	(s^2+2k s+\lambda^\alpha) x^\alpha\right] \,.
\label{dirichletsol}
\end{equation}
\end{prop}

\begin{proof}
	By taking the time-Laplace transform of \eqref{dr}, we have that
\begin{equation}
	s^2\tilde{u}+ 	2 ks \tilde{u} = 
e^{-\lambda x} D_x^\alpha [e^{\lambda x}\tilde{u}]-\lambda^\alpha \tilde{u} \,.
	\end{equation}
Therefore, we have that
\begin{equation}
	e^{-\lambda x} D_x^\alpha [e^{\lambda x}\tilde{u}] = \left(s^2+2ks+\lambda^\alpha\right)\tilde{u} \,,
\end{equation}
whose solution,
according to the boundary condition and by recalling that the one-parameter 
Mittag--Leffler function is an eigenfunction of the 
Caputo fractional derivative $D_x^\alpha$,
is given by (\ref{dirichletsol}).
	\end{proof}

We consider now the special case when $k=\lambda ^{\alpha /2}$,
then (\ref{dirichletsol}) can be rewritten as
\begin{equation}
\widetilde{u}(x,s)=\widetilde{\phi }(s)e^{-\lambda x}E_{\alpha
,1}\left[-(s+\lambda ^{\alpha /2})^{2}x^{\alpha }\right] \,,
\end{equation}
and the solution $u(x,t)$ can be explicitly derived.

We start by considering that
\begin{equation}
\int_{0}^{+\infty }e^{-\eta x}E_{\alpha ,1}(-\theta ^{2}x^{\alpha })dx=\frac{%
\eta ^{\alpha -1}}{\eta ^{\alpha }+\theta ^{2}} \,,  
\end{equation}
and we observe that its inverse Laplace transform 
with respect to $\theta $ is given by
\begin{equation}
\mathcal{L}^{-1}\left\{ \frac{\eta ^{\alpha -1}}{\eta ^{\alpha }+\theta ^{2}}%
;t\right\} =\eta ^{\alpha -1}tE_{2,2}(-\eta ^{\alpha }t^{2})
\,.
\end{equation}
Now we invert the Laplace transform with respect to $\eta $, 
by considering the following representation of 
the Mittag--Leffler function as H function
\cite[formula (1.136)]{mathai_etal-2010}: 
\begin{equation}
E_{\alpha ,\beta }(x)=H_{1,2}^{1,1}\left[ \left. -x\right\vert
\begin{array}{cc}
(0,1) & \; \\
(0,1) & (1-\beta ,\alpha)
\end{array}
\right] \,.
\end{equation}
We then apply the inverse transformation
\cite[formula (2.21)]{mathai_etal-2010}
(after checking that the conditions
are satisfied for $\sigma =\alpha $ and $\rho =1-\alpha $), as follows%
\begin{eqnarray*}
\mathcal{L}^{-1}\left\{ \eta ^{\alpha -1}t \,
E_{2,2} (- \eta^\alpha t^2) ; x\right\}  
&=&tx^{-\alpha }H_{2,2}^{1,1}\left[ \left. \frac{t^{2}}{x^{\alpha }}%
\right\vert
\begin{array}{cc}
(0,1) & (1-\alpha ,\alpha ) \\
(0,1) & (-1,2)%
\end{array}%
\right]  \\
&=&\cite[{\rm formula} \, (1.60)]{mathai_etal-2010} \\
&=&\frac{1}{t}H_{2,2}^{1,1}\left[ \left. \frac{t^{2}}{x^{\alpha }}%
\right\vert
\begin{array}{cc}
(1,1) & (1,\alpha ) \\
(1,1) & (1,2)%
\end{array}%
\right]  \\
&=&\cite[{\rm formula} \, (1.58)]{mathai_etal-2010} \\
&=&\frac{1}{t}H_{2,2}^{1,1}\left[ \left. \frac{x^{\alpha }}{t^{2}}%
\right\vert
\begin{array}{cc}
(0,1) & (0,2) \\
(0,1) & (0,\alpha )%
\end{array}%
\right]  \\
&=&\frac{1}{t}\frac{1}{2\pi i}\int_{L}\left( \frac{x^{\alpha }}{t^{2}}%
\right) ^{-w}\frac{\Gamma (w)\Gamma (1-w)}{\Gamma (2w)\Gamma (1-\alpha w)}
\, dw
\\
&=&[\text{by the duplication property of the Gamma function}] \\
&=&\frac{2\sqrt{\pi }}{t}\frac{1}{2\pi i}\int_{L}\left( \frac{x^{\alpha }}{%
4t^{2}}\right) ^{-w}\frac{\Gamma (1-w)}{\Gamma (\frac{1}{2}+w)\Gamma
(1-\alpha w)} \, dw \\
&=&\frac{2\sqrt{\pi }}{t}H_{2,1}^{0,1}\left[ \left. \frac{x^{\alpha }}{4t^{2}%
}\right\vert
\begin{array}{cc}
(0,1) & (1/2,1) \\
(0,\alpha ) & \;%
\end{array}%
\right] \,,
\end{eqnarray*}%
where $L$ is the loop beginning and ending at $+\infty $, 
denoted by $L_{+\infty }$ in the treatise
\cite[point ii), p. 3]{mathai_etal-2010}, 
since, in this case, $\mu =\alpha -2>0$.

As a consequence of the previous steps we can write the solution 
$u(x,t)$ as follows
\begin{equation}
u(x,t)=2\sqrt{\pi }e^{-\lambda x}\int_{0}^{t}\phi (t-z)\frac{e^{-\lambda
^{\alpha /2}z}}{z}H_{2,1}^{0,1}\left[ \left. \frac{x^{\alpha }}{4z^{2}}%
\right\vert
\begin{array}{cc}
(0,1) & (1/2,1) \\
(0,\alpha ) & \;%
\end{array}%
\right] \, dz \,.
\end{equation}

\section*{Acknowledgments}
GP is supported by the Basque Government through the 2022--2025 programs 
and by the Ministry of Science, Innovation and Universities: 
BCAM Severo Ochoa accreditation SEV-2017-0718.
The research was carried out under the auspices of 
INDAM-GNFM 
(the National Group of Mathematical Physics of 
the Italian National Institute of High Mathematics).


\end{document}